\documentclass[12pt]{amsart}
\usepackage{graphicx}
\usepackage{a4wide}
\usepackage{amssymb}
\usepackage{amsthm}
\usepackage{amsmath}
\usepackage{amscd}
\usepackage{verbatim}
\usepackage{nicefrac}
\usepackage[all]{xy}
\usepackage{tikz}
\usepackage{enumerate}
\usepackage{url}
\textheight8.5in \textwidth6.6in \numberwithin{equation}{section}
\theoremstyle{plain}
\newtheorem{theorem}{Theorem}[section]
\newtheorem{corollary}[theorem]{Corollary}
\newtheorem{lemma}[theorem]{Lemma}

\theoremstyle{definition}
\newtheorem{definition}[theorem]{Definition}
\newtheorem*{example}{Example}

\theoremstyle{remark}
\newtheorem*{remark}{Remark}

\newcommand{\SL}{\text {\rm SL}}

\newcommand{\Q}{\mathbb{Q}}
\newcommand{\prim}{\text {\rm prim}}
\newcommand{\parti}{\text {\rm part}}

\newcommand{\Z}{\mathbb{Z}}

\newcommand{\C}{\mathbb{C}}

\renewcommand{\H}{\mathbb{H}}
\newcommand{\F}{\mathbb{F}}

\def\O{\mathcal O}

\newcommand{\im}{\mbox{Im}}

\newcommand{\calQ}{\mathcal{Q}}

\newcommand{\fraka}{\mathfrak a}

\newcommand{\End}{\operatorname{End}}

\newcommand{\Gal}{\operatorname{Gal}}

\newcommand{\cl}{{\rm cl}}
\newcommand{\Ell}{{\rm Ell}}
\newcommand{\Fp}{\F_p}
\newcommand{\disc}{\operatorname{disc}}
\def\ndiv{\hbox{${ }\not | \ { }$}}
\newcommand{\inkron}[2]{\genfrac {(}{)}{0.9pt}{}{#1}{#2}}
\newcommand{\height}{{\rm ht}}
\newcommand{\softO}{\widetilde O}
\newcommand{\jt}{{\tilde\jmath}}

\theoremstyle{plain}
\newtheorem*{gamma_thm}{\rm\textbf{Theorem \ref{gamma_thm}}}
\newtheorem*{general_thm}{\rm\textbf{Theorem \ref{general_thm}}}
\newtheorem*{partition_thm}{\rm\textbf{Theorem \ref{partition_thm}}}

\begin{document}

\title[ class polynomials for nonholomorphic modular functions] { class polynomials for nonholomorphic\\ modular functions}

\author{Jan Hendrik Bruinier, Ken Ono, and Andrew V. Sutherland}
\address{Fachbereich Mathematik, Technische Universit\"at Darmstadt,
Schlossgartenstrasse 7, D-64289, Darmstadt, Germany}
\email{bruinier@mathematik.tu-darmstadt.de}

\address{Department of Mathematics and Computer Science,
Emory University, Atlanta, GA 30322}
\email{ono@mathcs.emory.edu}

\address{Department of Mathematics, Massachusetts Institute of Technology, Cambridge, MA 02139}
\email{drew@math.mit.edu}

\thanks{The first author was supported by DFG grant BR-2163/2-2.
The second author  thanks the support of NSF grant DMS -1157289 and the
Asa Griggs Candler Fund.
The third author received support from NSF grant DMS-1115455}.

\begin{abstract} We give algorithms for computing the {\it singular moduli} of suitable nonholomorphic modular functions $F(z)$.
By combining the theory of {\it isogeny volcanoes} with a beautiful observation of Masser concerning the nonholomorphic Eisenstein series
$E_2^{*}(z)$, we obtain CRT-based algorithms that compute the class polynomials $H_D(F;x)$, whose roots are the discriminant~$D$ singular moduli for $F(z)$. By applying these results to a specific weak Maass form  $F_p(z)$,
we obtain a CRT-based algorithm for computing {\it partition class polynomials}, a sequence of polynomials whose traces give the partition numbers $p(n)$. Under the GRH, the expected running time of this algorithm is $O(n^{5/2+o(1)})$.
Key to these results is a fast CRT-based algorithm for computing the classical modular polynomial $\Phi_m(X,Y)$ that we obtain by extending the isogeny volcano approach previously developed for prime values of $m$.
 \end{abstract}

\maketitle

\section{Introduction and Statement of results}

As usual, we let
\begin{equation}
j(z):=\frac{\left(1+240\sum_{n=1}^{\infty}\sum_{d\mid n}d^3q^n\right)^3}{q\prod_{n=1}^{\infty}(1-q^n)^{24}}=q^{-1}+744+196884q+21493760q^2+\cdots
\end{equation}
be Klein's classical elliptic modular function on $\operatorname{SL}_2(\Z)$
($q:=e^{2\pi i z}$ throughout). The values of $j(z)$ at imaginary
quadratic arguments in the upper-half of the complex plane are
known as {\it singular moduli}.
Two examples are the Galois conjugates
\begin{displaymath}
 j\left(\frac{1+\sqrt{-15}}{2}\right)=\frac{-191025-85995\sqrt{5}}{2} \ \ \ \
{\text {\rm and}}\ \ \ \ j\left(\frac{1+\sqrt{-15}}{4}\right)=\frac{-191025+85995\sqrt{5}}{2}.
\end{displaymath}
These numbers play an important role in algebraic number theory.
Indeed, they are algebraic integers that
generate ring
class field extensions of imaginary quadratic fields, and they are the $j$-invariants
of elliptic curves with complex multiplication \cite{Borel, Cox, D1, D2}.

The problem of computing singular moduli has a long history
that dates back to the works of Kronecker, and is highlighted by famous
calculations
by Berwick \cite{Be} and Weber \cite{W}. Historically, these numbers have been difficult to compute.
More recently, Gross and Zagier \cite{G-Z} determined the
prime factorization of the absolute norm of suitable differences of
singular moduli (further work in this direction has been
carried out by Dorman \cite{Do1, Do2}), and Zagier \cite{Zagier} identified
the algebraic traces of singular moduli as coefficients of half-integral weight modular forms.

Here we consider the problem
of computing the minimal polynomials of singular moduli,
the so-called Hilbert class polynomials. This problem has been the subject of much recent study. For example,
Belding, Br\"oker, Enge, and Lauter \cite{BBEL}, and the third
author \cite{Sutherland}, have provided efficient methods of computation that are based on the theory of elliptic curves with
complex multiplication (CM).  The basic approach in that work is simple. One uses theoretical facts about elliptic curves with CM to quickly compute the reductions of these polynomials modulo a set of suitable primes $p$, and one then compiles these reductions via the Chinese Remainder Theorem
(CRT) to obtain the exact polynomials. The primes $p$ are chosen in a way that facilitates the computation, and in particular, they split completely in the ring class field $K_\O$ of the imaginary quadratic order~$\O$ associated to the singular moduli whose minimal polynomial one wishes to compute.
Under the Generalized Riemann Hypothesis (GRH), this
algorithm computes the discriminant $D$ Hilbert class polynomial with an expected running time of $\softO(|D|)$.\footnote{We use the ``soft" asymptotic notation $\softO(n)$ to denote bounds of the form $O(n\log^c n)$.}
In \cite{BLS}, the third author, together with Br\"oker and Lauter, further developed these ideas to compute
modular polynomials $\Phi_\ell$ using the theory of {\it isogeny volcanoes}.

We extend these results to the setting
of nonholomorphic modular functions such as
\begin{equation}\label{gammaz}
\gamma(z):=\frac{E_4(z)}{6E_6(z)j(z)}\cdot E_2^{*}(z)-\frac{7j(z)-6912}{6j(z)(j(z)-1728)},
\end{equation}
where
\begin{equation}\label{E2}
E_2^*(z):=1-\frac{3}{\pi \im(z)}-24\sum_{n=1}^{\infty}
\sum_{d\mid n}dq^n
\end{equation}
is the weight 2 nonholomorphic modular Eisenstein series, and where
\begin{displaymath}
E_4(z):=1+240\sum_{n=1}^{\infty}\sum_{d\mid n}d^3q^n\ \ \ \ \ {\text {\rm and}}\ \ \
E_6(z):=1-504\sum_{n=1}^{\infty}\sum_{d\mid n}d^5q^n
\end{displaymath}
are the usual weight 4 and weight 6 modular Eisenstein series.

\begin{remark} The function $\gamma(z)$ plays an important role in this paper. We shall make use of an observation of Masser \cite{Masser} which
gives a description of its singular moduli in terms
of the coefficients of certain representations of the classical modular polynomials.
\end{remark}

We first recall the setting of Heegner points on modular curves (see
\cite{GKZ}).  Let $N>1$, and let $D<0$ be a quadratic discriminant coprime to $N$. The group $\Gamma_0(N)$ acts on the
discriminant~$D$ positive definite integral binary quadratic forms
$$
Q(X,Y)=[a,b,c]:=aX^2+bXY+cY^2
$$
with $N\mid a$.  This action preserves $b\pmod{2N}$. Therefore, if
$\beta^2\equiv D\pmod{4N}$, then it is natural to consider
$\calQ_{N,D,\beta}$, the set of those discriminant $D$ forms
$Q=[a,b,c]$ for which $0<a\equiv 0\pmod{N}$ and $b\equiv
\beta\pmod{2N}$, and we may also consider the subset $\calQ_{N,D,\beta}^{\prim}$ obtained by restricting to primitive forms.
The number of $\Gamma_0(N)$ equivalence classes in
$\calQ_{N,D,\beta}$ is the Hurwitz-Kronecker class number $H(D)$, and the natural map
defines a bijection
$$
\calQ_{N,D,\beta}/\Gamma_0(N)\longrightarrow \calQ_D/\SL_2(\Z),
$$
where $\calQ_D$  is the set of discriminant
$D$ positive definite integral binary quadratic
forms
(see the proposition on p.~505 of \cite{GKZ}).
This bijection also holds when restricting to primitive forms, in which case the number of $\Gamma_0(N)$ equivalence classes in $\calQ_{N,D,\beta}^{\prim}$, and the number of $\SL_2(\Z)$ equivalence classes in $\calQ_D^{\prim}$, is given by the class number $h(D)$.

For modular functions $F(z)$ on $\SL_2(\Z)$, our  goal is to calculate
the \emph{class polynomial}
\begin{equation}\label{classpoly1}
H_D(F;x):=\prod_{Q\in \calQ_D^{\prim}/\SL_2(\Z)}\bigl(x-F(\alpha_Q)\bigr),
\end{equation}
where $\alpha_Q\in \H$ is a root of $Q(x,1)=0$.
For modular functions $F(z)$ on $\Gamma_0(N)$ and discriminants $D<0$ coprime to $N$, our goal is to calculate
the {\it class polynomial}
\begin{equation}
\label{classpolyN}
H_{D,\beta}(F;x):=\prod_{Q\in \calQ_{N,D,\beta}^{\prim}/\Gamma_0(N)}\bigl(x-F(\alpha_Q)\bigr).
\end{equation}

We first consider the
special case of the $\SL_2(\Z)$ nonholomorphic modular function $\gamma(z)$,
defined by \eqref{gammaz}, and we compute the $\Q$-rational
polynomials $H_D(\gamma;x)$.

\begin{theorem}
\label{gamma_thm}
For discriminants $D<-4$ that are not of the form
$D=-3d^2$, Algorithm 1 (see \S\ref{gamma_case}) computes $H_D(\gamma;x)$.
Under the GRH, its expected running time is $\softO(|D|^{7/2})$ and it uses $\softO(|D|^2)$ space.
\end{theorem}

\begin{remark} The discussion on p. 118 of \cite{Masser}, makes it clear how to modify Algorithm 1 to handle discriminants of the form $D=-3d^2$.  We expect that the bound on its expected running time can be improved to $\softO(|D|^{5/2})$ using tighter bounds on the size of the coefficients of $H_D(\gamma;x)$.
\end{remark}

A key building block of Algorithm 1 is a new algorithm to compute the classical modular polynomial $\Phi_m(X,Y)$, which parameterizes pairs of elliptic curves related by a cyclic isogeny of degree $m$.
Here we extend the iosgeny volcano approach that was introduced in \cite{BLS} to compute $\Phi_m$ for prime $m$ so that we can now efficiently handle all values of $m$.
The result is Algorithm~1.1 (see \S\ref{gamma_case}), which, under the GRH, computes $\Phi_m$ in $\tilde{O}(m^3)$ time.
For suitable primes~$p$ it can compute $\Phi_m$ modulo $p$ in $\tilde{O}(m^2)$ time (and space), which is crucial to the efficient implementation of Algorithm~1.

\begin{example}
We have used Algorithm 1 to compute $H_D(\gamma;x)$ for $D>-20000$.  Some small examples are listed below.

\bigskip
\begin{center}
\begin{tabular}{cccccc}
\hline\hline\vspace{-8pt} \ \ \\
$D$ & $H_D(\gamma;x)$
\\ \hline \\
$-3$  & $x - \frac{23}{2^{11}\cdot 3^3}$\\ \ \ \\
$-4$  & $x$\\ \ \ \\
$-7$  & $x - \frac{181}{3^6\cdot 5^3\cdot 7}$\\ \ \ \\
$-8$  & $x + \frac{61}{2^6\cdot 5^3\cdot 7^2}$\\ \ \ \\
$-11$& $x - \frac{289}{2^{14}\cdot 7^2\cdot 11}$\\ \ \ \\
$-12$& $x + \frac{67}{2^3\cdot 3^3\cdot 5^3\cdot 11^2}$\\ \ \ \\
$-15$& $x^2 + \frac{313}{3^4\cdot 5\cdot 11^3}\cdot x - \frac{1045769}{3^8\cdot 5^3\cdot 7^4\cdot 11^5}$\\ \ \ \\
$-16$& $x + \frac{179}{3^6\cdot 7^2\cdot 11^3}$\\ \ \ \\
$-19$& $x - \frac{275}{2^{14}\cdot 3^6\cdot 19}$\\ \ \ \\
$-20$& $x^2 - \frac{43925}{2^6\cdot 11^3\cdot 19^2}\cdot x - \frac{2307859}{2^{18}\cdot 5^3\cdot 11^5\cdot 19^2}$\\ \ \ \\
$-23$&\ \ $x^3 + \frac{8123835989}{5^3\cdot 7^2\cdot 11^3\cdot 17^3\cdot 19^2\cdot 23}
\cdot x^2 + \frac{6062055706222}{5^6\cdot 7^4\cdot 11^4\cdot 17^3\cdot 19^2\cdot 23}\cdot
x - \frac{346923509992369}{5^9\cdot 7^6\cdot 11^7\cdot 17^3\cdot 19^2\cdot 23^2}$\\\vspace{-8pt} \ \ \\
\hline\hline
\end{tabular}
\end{center}
\medskip
\end{example}

We extend Algorithm 1 to compute class polynomials for a large class of
nonholomorphic modular functions.  This class includes, for example, the $\SL_2(\Z)$-function
$$
K(z):=288\cdot \frac{E_2^{*}(z)E_4(z)E_6(z)+3E_4(z)^3+2E_6(z)^2}{E_4(z)^3-E_6(z)^2}
$$
considered by Zagier (see \S 9 of \cite{Zagier}) in his famous paper on traces of
singular moduli.
More generally, it includes suitable modular functions $F(z)$ of the form
$$
F(z):=\partial_{-2}\circ \partial_{-4}\cdots \partial_{2-2k}(\mathfrak{F}),
$$
where $\mathfrak{F}(z)$ is a weight $2-2k$ weakly holomorphic modular form on $\Gamma_0(N)$
whose Fourier expansions at cusps are algebraic. Here the differential operator
$\partial_h$, which maps weight $h$ modular forms to weight $h+2$ modular forms, is defined by
\begin{equation}\label{del}
\partial_h:=\frac{1}{2\pi i}\cdot \frac{\partial}{\partial z}-\frac{h}{4\pi \im(z)}.
\end{equation}

We consider a specific class of such modular functions.
Let $\O$ be the imaginary quadratic order with discriminant $D$, ring class field $K_\O$, and fraction field $K=\Q(\sqrt{D})$.
Let $c_1$ and~$c_2$ denote fixed positive integers.
Let $F(z)$ be a modular function (i.e. weight 0) for $\Gamma_0(N)$ that can be written in the form
\[
F(z)=\sum A_n(z)\gamma(z)^n,
\]
where each $A_n\in \Q(j)$ is a rational function of $j(z)$.
The function $F(z)$ is said to be {\bf good} for a discriminant $D<0$ coprime to $N$ if it satisfies the following:
\begin{enumerate}
\item Each $A_n(\alpha_Q)$ lies in $K_\O$ for all $Q\in \calQ_D^{\prim}/\SL_2(\Z)$.
\item The polynomial $c_1|D|^{c_2h}H_D(F;x)$ has integer coefficients.
\end{enumerate}

\begin{remark}
The class of good modular functions includes many nonholomorphic modular functions, and it includes meromorphic modular functions which may have poles in the upper half plane (poles at CM points are excluded by condition (1)).
\end{remark}

For good modular functions, we obtain the following general result.

\begin{theorem}\label{general_thm}  For discriminants $D<-4$ not of the form
$D=-3d^2$, Algorithm~2 (see \S\ref{good_case}) computes class polynomials for good modular functions~$F(z)$.  Under the GRH, its expected running time is $\softO(|D|^{\nicefrac{5}{2}})$, and it uses $\softO(|D|^2)$ space.
\end{theorem}

In fact, Algorithm~2 can be readily adapted to treat modular functions of the form $F(z)=\sum A_n(z)\gamma(z)^n$ where the coefficient functions $A_n(z)$ do not necessarily lie in $\Q(j)$, using the techniques developed by Enge and the third author in \cite{EngeSutherland}.
As an example, we apply Algorithm~2 to obtain a CRT-based algorithm for computing the ``partition polynomials"
defined by the first two authors in \cite{BruinierOno2011}. These are essentially the class polynomials
of the $\Gamma_0(6)$ nonholomorphic modular function
\begin{equation}\label{Fp}
F_p(z):=-\partial_{-2}(P(z))=\left(1-\frac{1}{2\pi \im(z)}\right)q^{-1}+
\frac{5}{\pi \im(z)}+\left(29+\frac{29}{2\pi \im(z)}\right)q+\dots,
\end{equation}
where $P(z)$ is the weight $-2$ weakly holomorphic modular form
\begin{displaymath}
P(z):=\frac{1}{2}\cdot \frac{E_2(z)-2E_2(2z)-3E_2(3z)+6E_2(6z)}
{\eta(z)^2\eta(2z)^2\eta(3z)^2\eta(6z)^2}=q^{-1}-10-29q+\dots.
\end{displaymath}
These polynomials are defined as
\begin{equation}\label{partitionpolynomial}
H_{n}^{\parti}(x):=\prod_{Q\in \calQ_{6,1-24n,1}}(x-P(\alpha_Q)).
\end{equation}
In contrast with (\ref{classpoly1}) and (\ref{classpolyN}), we stress that the roots of these polynomials
include singular moduli for imprimitive forms (if any).
The interest in these polynomials arises from the
fact that (see Theorem 1.1 of \cite{BruinierOno2011})
\begin{equation}\label{ptrace}
H_{n}^{\parti}(x)=x^{h(1-24n)}-(24n-1)p(n)x^{h(1-24n)-1}+\dots,
\end{equation}
where $p(n)$ is the usual partition function. Since the roots $P(\alpha_Q)$
are algebraic numbers that lie in
the usual discriminant $1-24n$ ring class field,
we have the following finite algebraic formula
$$
p(n)=\frac{1}{24n-1}\sum_{Q\in \calQ_{6,1-24n,1}} P(\alpha_Q).
$$

\begin{remark}
By the work of the first two authors \cite{BruinierOno2011}, combined with recent results by Larson and Rolen \cite{LarsonRolen}, it is known  that each $(24n-1)P(\alpha_Q)$ is an
algebraic integer.
\end{remark}

As a consequence of Theorem~\ref{general_thm}, we obtain the following result.

\begin{theorem}\label{partition_thm} For all positive integers $n$, Algorithm 3 (see \S\ref{P_case}) computes $H_{n}^{\parti}(x)$.
Under the GRH, its expected running time is $\softO(n^{\nicefrac{5}{2}})$ and it uses $\softO(n^2)$ space.
\end{theorem}

\begin{remark}
For the simpler task of computing individual values of $p(n)$, an efficient implementation of Rademacher's formula such as the one given in \cite{Johansson} is both asymptotically and practically faster than Algorithm~3.
\end{remark}

\begin{example} We have used Algorithm 3 to compute $H_{n}^{\parti}(x)$ for $n\leq 750$.  Some small examples are listed below.

\bigskip
\begin{center}
\begin{tabular}{cccccc}
\hline\hline\vspace{-8pt} \ \ \\
$n$ && $(24n-1)p(n)$ && $H_{n}^{\parti}(x)$
\\ \hline \\
$1$ &&  $23$ && $x^3-23x^2+\frac{3592}{23}x-419$  \\ \ \ \\
$2$ &&  $94$ && $x^5-94x^4+\frac{169659}{47}x^3-65838x^2+\frac{1092873176}{47^2}x
+\frac{1454023}{47}$  \\ \ \ \\
$3$ && $213$ && $x^7-213x^6+\frac{1312544}{71}x^5-723721x^4+\frac{44648582886}{71^2}x^3$\\
\ \ \\
\ \ && \ \ &&\ \ \ \ \ \ \ \ \ \ \ \ \ $+\frac{9188934683}{71}x^2+
\frac{166629520876208}{71^3}x+\frac{2791651635293}{71^2}$\\
\ \ \\
$4$ && $475$ && $x^8-475x^7+\frac{9032603}{95}x^6-9455070x^5+
\frac{3949512899743}{95^2}x^4$\\ \ \ \\
\ \ && \ \ && \ \ \ \ \ \ \ $-\frac{97215753021}{19}x^3+\frac{9776785708507683}{95^3}x^2$
\\ \ \ \\
\ \ && \ \ && \ \ \ \ \ \ \ \ \ \ \ \ \ \ \ \ \ \ \ $-\frac{53144327916296}{19^2}x-
\frac{134884469547631}{5^4\cdot 19}$\\
\vspace{-8pt} \ \ \\
\hline\hline
\end{tabular}
\end{center}
\end{example}

\bigskip
This paper is organized as follows. In \S\ref{nutsandbolts} we recall
essential facts about elliptic curves with complex multiplication, and
singular moduli for modular forms and certain nonholomorphic modular functions.
In \S\ref{alg} we use these results to derive our algorithms.
In \S\ref{example} we conclude with a detailed example of the execution
of Algorithm~3 for $n=1$ and $n=24$.

\section{nuts and bolts}\label{nutsandbolts}

We begin with some preliminaries on elliptic curves with complex multiplication
and singular moduli for suitable modular functions.

\subsection{Elliptic curves with complex multiplication}
We recall some standard facts from the theory of complex multiplication, referring to \cite{Cox,Lang,Serre} for proofs and further background.
Let~$\O$ be an imaginary quadratic order, identified by its discriminant~$D$.
The $j$-invariant of the lattice~$\O$ is an algebraic integer whose minimal polynomial is the \emph{Hilbert class polynomial}~$H_D$.
If~$\fraka$ is an invertible $\O$-ideal (including $\fraka=\O$), then the torus $\C/\fraka$ corresponds to an elliptic curve $E/\C$ with \emph{complex multiplication} (CM) by~$\O$, meaning that its endomorphism ring $\End(E)$ is isomorphic to $\O$, and every such curve arises in this fashion.
Equivalent ideals yield isomorphic elliptic curves, and this gives a bijection between the ideal class group $\cl(\O)$ and the set
\begin{equation}
\Ell_\O(\C) := \{j(E/\C)\colon \End(E)\cong\O\},
\end{equation}
the $j$-invariants of the elliptic curves defined over $\C$ with CM by $\O$.
We then have
\begin{equation}\label{eq:class1}
H_D(x) := \prod_{j_i\in\Ell_\O(\C)}(x-j_i) = H_D(j;x).
\end{equation}
The splitting field of $H_D$ over $K=\Q(\sqrt{D})$ is the \emph{ring class field} $K_\O$.
It is an abelian extension of $K$ whose Galois group $\Gal(K_\O/K)$ is isomorphic to $\cl(\O)$, via the Artin map.

This isomorphism can be made explicit via isogenies.
Let $E/\C$ be an elliptic curve with CM by $\O$ and let $\fraka$ be an invertible $\O$-ideal.
There is a uniquely determined separable isogeny whose kernel is the subgroup of points annihilated by every endomorphism in $\fraka\subset\O\hookrightarrow\End(E)$.
The image of this isogeny is an elliptic curve that also has CM by~$\O$, and this defines an action of the ideal group of~$\O$ on the set $\Ell_\O(\C)$.
Principal ideals act trivially, and the induced action of the class group is regular.
Thus $\Ell_\O(\C)$ is a principal homogeneous space, a \emph{torsor}, for the finite abelian group $\cl(\O)$.

If $p$ is a (rational) prime that splits completely in $K_\O$, equivalently, for $p>3$, a prime satisfying
the \emph{norm equation}
\begin{equation}\label{eq:norm}
4p = t^2-v^2D
\end{equation}
for some nonzero integers $t$ and $v$, then $H_D$ splits completely in $\Fp[x]$ and its roots form the set
\begin{equation}
\Ell_\O(\Fp) := \{j(E/\Fp)\colon \End(E)\cong\O\},
\end{equation}
Conversely, every ordinary (not supersingular) elliptic curve $E/\Fp$ has CM by some imaginary quadratic order $\O$ in which the Frobenius endomorphism corresponds to an element of norm~$p$ and trace~$t$.

\subsection{Modular polynomials via isogeny volcanoes}
For each positive integer $m$, the classical modular polynomial $\Phi_m$ is the minimal polynomial of the function $j(mz)$ over the field $\C(j)$.
As a polynomial in two variables, $\Phi_m\in\Z[X,Y]$ is symmetric in $X$ and $Y$.
If $E/k$ is an elliptic curve and $m$ is prime to the characteristic of $k$, then the roots of $\Phi_m(j(E),Y)$ are precisely the $j$-invariants of the elliptic curves  that are related to $E$ by a cyclic $m$-isogeny; see~\cite{Igusa,Lang} for these and other properties of~$\Phi_m$.

For distinct primes $\ell$ and~$p$, we define the \emph{graph of $\ell$-isogenies} $G_\ell(\Fp)$, with vertex set $\Fp$ and edges $(j_1,j_2)$ present if and only if $\Phi_\ell(j_1,j_2)=0$.
Ignoring the connected components of $0$ and $1728$, the ordinary components of $G_\ell(\Fp)$ are $\ell$-\emph{volcanoes} \cite{FM,Kohel}, a term we take to include cycles as a special case; see \cite{Sutherland2} for further details on isogeny volcanoes.
In this paper we focus on $\ell$-volcanoes of a special form, for which we can compute $\Phi_\ell\bmod p$ in a particularly efficient way, using \cite[Alg.~2.1]{BLS}.

Let $\O$ be an order in an imaginary quadratic field $K$ with maximal order $\O_K$, and let $\ell$ be an odd prime not dividing $[\O_K:\O]$. Assume $D=\disc(\O)<-4$.
Suppose $p$ is a prime of the form $4p=t^2-\ell^2v^2D$ with $p\equiv 1\bmod \ell$ and $\ell\ndiv v$; equivalently, $p$ splits completely in the ray class field of conductor $\ell$ for $\O$ and does not split completely in the ring class field of the order with index $\ell^2$ in $\O$.
Then the components of $G_\ell(\Fp)$ that intersect $\Ell_\O(\Fp)$ are isomorphic $\ell$-volcanoes with two levels: the \emph{surface}, whose vertices lie in $\Ell_\O(\Fp)$, and the \emph{floor}, whose vertices lie in $\Ell_{\O'}(\Fp)$, where $\O'$ is the order of index $\ell$ in $\O$.
Each vertex on the surface is connected to $1+\inkron{D}{\ell}=0,1$ or $2$ \emph{siblings} on the surface, and $\ell-\inkron{D}{\ell}$ \emph{children} on the floor.
An example with $\ell=7$ is shown below:

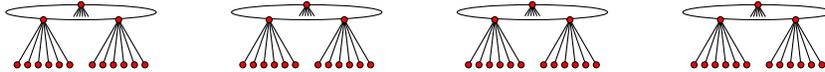
\begin{figure}[htp]
\begin{tikzpicture}
\draw (-4.5,0) ellipse (1 and 0.1);
\draw (-1.5,0) ellipse (1 and 0.1);
\draw (1.5,0) ellipse (1 and 0.1);
\draw (4.5,0) ellipse (1 and 0.1);
\draw (-4.5,0.1) -- (-4.6,-0.06);
\draw (-4.5,0.1) -- (-4.56,-0.06);
\draw (-4.5,0.1) -- (-4.52,-0.06);
\draw (-4.5,0.1) -- (-4.48,-0.06);
\draw (-4.5,0.1) -- (-4.44,-0.06);
\draw (-4.5,0.1) -- (-4.4,-0.06);
\draw[fill=red] (-4.5,0.1) circle (0.04);
\draw (-1.5,0.1) -- (-1.6,-0.05);
\draw (-1.5,0.1) -- (-1.56,-0.05);
\draw (-1.5,0.1) -- (-1.52,-0.05);
\draw (-1.5,0.1) -- (-1.48,-0.05);
\draw (-1.5,0.1) -- (-1.44,-0.05);
\draw (-1.5,0.1) -- (-1.4,-0.05);
\draw[fill=red] (-1.5,0.1) circle (0.04);
\draw (1.5,0.1) -- (1.6,-0.05);
\draw (1.5,0.1) -- (1.56,-0.05);
\draw (1.5,0.1) -- (1.52,-0.05);
\draw (1.5,0.1) -- (1.48,-0.05);
\draw (1.5,0.1) -- (1.44,-0.05);
\draw (1.5,0.1) -- (1.4,-0.05);
\draw[fill=red] (1.5,0.1) circle (0.04);
\draw (4.5,0.1) -- (4.60,-0.06);
\draw (4.5,0.1) -- (4.56,-0.06);
\draw (4.5,0.1) -- (4.52,-0.06);
\draw (4.5,0.1) -- (4.48,-0.06);
\draw (4.5,0.1) -- (4.44,-0.06);
\draw (4.5,0.1) -- (4.4,-0.06);
\draw[fill=red] (4.5,0.1) circle (0.04);
\draw (-5,-0.1) -- (-5.35,-0.7);
\draw[fill=red] (-5.35,-0.7) circle (0.04);
\draw (-5,-0.1) -- (-5.21,-0.7);
\draw[fill=red] (-5.21,-0.7) circle (0.04);
\draw (-5,-0.1) -- (-5.07,-0.7);
\draw[fill=red] (-5.07,-.7) circle (0.04);
\draw (-5,-0.1) -- (-4.93,-0.7);
\draw[fill=red] (-4.93,-0.7) circle (0.04);
\draw (-5,-0.1) -- (-4.79,-0.7);
\draw[fill=red] (-4.79,-0.7) circle (0.04);
\draw (-5,-0.1) -- (-4.65,-0.7);
\draw[fill=red] (-4.65,-0.7) circle (0.04);
\draw[fill=red] (-5,-0.1) circle (0.04);

\draw (-4,-0.1) -- (-4.35,-0.7);
\draw[fill=red] (-4.35,-0.7) circle (0.04);
\draw (-4,-0.1) -- (-4.21,-0.7);
\draw[fill=red] (-4.21,-0.7) circle (0.04);
\draw (-4,-0.1) -- (-4.07,-0.7);
\draw[fill=red] (-4.07,-0.7) circle (0.04);
\draw (-4,-0.1) -- (-3.93,-0.7);
\draw[fill=red] (-3.93,-0.7) circle (0.04);
\draw (-4,-0.1) -- (-3.79,-0.7);
\draw[fill=red] (-3.79,-0.7) circle (0.04);
\draw (-4,-0.1) -- (-3.65,-0.7);
\draw[fill=red] (-3.65,-0.7) circle (0.04);
\draw[fill=red] (-4,-0.1) circle (0.04);

\draw (-2,-0.1) -- (-2.35,-0.7);
\draw[fill=red] (-2.35,-0.7) circle (0.04);
\draw (-2,-0.1) -- (-2.21,-0.7);
\draw[fill=red] (-2.21,-0.7) circle (0.04);
\draw (-2,-0.1) -- (-2.07,-0.7);
\draw[fill=red] (-2.07,-0.7) circle (0.04);
\draw (-2,-0.1) -- (-1.93,-0.7);
\draw[fill=red] (-1.93,-0.7) circle (0.04);
\draw (-2,-0.1) -- (-1.79,-0.7);
\draw[fill=red] (-1.79,-0.7) circle (0.04);
\draw (-2,-0.1) -- (-1.65,-0.7);
\draw[fill=red] (-1.65,-0.7) circle (0.04);
\draw[fill=red] (-2,-0.1) circle (0.04);

\draw (-1,-0.1) -- (-1.35,-0.7);
\draw[fill=red] (-1.35,-0.7) circle (0.04);
\draw (-1,-0.1) -- (-1.21,-0.7);
\draw[fill=red] (-1.21,-0.7) circle (0.04);
\draw (-1,-0.1) -- (-1.07,-0.7);
\draw[fill=red] (-1.07,-0.7) circle (0.04);
\draw (-1,-0.1) -- (-0.93,-0.7);
\draw[fill=red] (-0.93,-0.7) circle (0.04);
\draw (-1,-0.1) -- (-0.79,-0.7);
\draw[fill=red] (-0.79,-0.7) circle (0.04);
\draw (-1,-0.1) -- (-0.65,-0.7);
\draw[fill=red] (-0.65,-0.7) circle (0.04);
\draw[fill=red] (-1,-0.1) circle (0.04);

\draw (1,-0.1) -- (1.35,-0.7);
\draw[fill=red] (1.35,-0.7) circle (0.04);
\draw (1,-0.1) -- (1.21,-0.7);
\draw[fill=red] (1.21,-0.7) circle (0.04);
\draw (1,-0.1) -- (1.07,-0.7);
\draw[fill=red] (1.07,-0.7) circle (0.04);
\draw (1,-0.1) -- (0.93,-0.7);
\draw[fill=red] (0.93,-0.7) circle (0.04);
\draw (1,-0.1) -- (0.79,-0.7);
\draw[fill=red] (0.79,-0.7) circle (0.04);
\draw (1,-0.1) -- (0.65,-0.7);
\draw[fill=red] (0.65,-0.7) circle (0.04);
\draw[fill=red] (1,-0.1) circle (0.04);

\draw (2,-0.1) -- (2.35,-0.7);
\draw[fill=red] (2.35,-0.7) circle (0.04);
\draw (2,-0.1) -- (2.21,-0.7);
\draw[fill=red] (2.21,-0.7) circle (0.04);
\draw (2,-0.1) -- (2.07,-0.7);
\draw[fill=red] (2.07,-0.7) circle (0.04);
\draw (2,-0.1) -- (1.93,-0.7);
\draw[fill=red] (1.93,-0.7) circle (0.04);
\draw (2,-0.1) -- (1.79,-0.7);
\draw[fill=red] (1.79,-0.7) circle (0.04);
\draw (2,-0.1) -- (1.65,-0.7);
\draw[fill=red] (1.65,-0.7) circle (0.04);
\draw[fill=red] (2,-0.1) circle (0.04);

\draw (4,-0.1) -- (4.35,-0.7);
\draw[fill=red] (4.35,-0.7) circle (0.04);
\draw (4,-0.1) -- (4.21,-0.7);
\draw[fill=red] (4.21,-0.7) circle (0.04);
\draw (4,-0.1) -- (4.07,-0.7);
\draw[fill=red] (4.07,-0.7) circle (0.04);
\draw (4,-0.1) -- (3.93,-0.7);
\draw[fill=red] (3.93,-0.7) circle (0.04);
\draw (4,-0.1) -- (3.79,-0.7);
\draw[fill=red] (3.79,-0.7) circle (0.04);
\draw (4,-0.1) -- (3.65,-0.7);
\draw[fill=red] (3.65,-0.7) circle (0.04);
\draw[fill=red] (4,-0.1) circle (0.04);

\draw (5,-0.1) -- (5.35,-0.7);
\draw[fill=red] (5.35,-0.7) circle (0.04);
\draw (5,-0.1) -- (5.21,-0.7);
\draw[fill=red] (5.21,-0.7) circle (0.04);
\draw (5,-0.1) -- (5.07,-0.7);
\draw[fill=red] (5.07,-0.7) circle (0.04);
\draw (5,-0.1) -- (4.93,-0.7);
\draw[fill=red] (4.93,-0.7) circle (0.04);
\draw (5,-0.1) -- (4.79,-0.7);
\draw[fill=red] (4.79,-0.7) circle (0.04);
\draw (5,-0.1) -- (4.65,-0.7);
\draw[fill=red] (4.65,-0.7) circle (0.04);
\draw[fill=red] (5,-0.1) circle (0.04);
\end{tikzpicture}
\end{figure}

Provided $h(\O)\ge \ell+2$, this set of $\ell$-volcanoes contains enough information to completely determine $\Phi_\ell\bmod p$.
This is the basis of the algorithm in \cite[Alg.\ 2.1]{BLS} to compute $\Phi_\ell\bmod p$, which we make use of here.
Selecting a sufficiently large set of such primes $p$ allows one to compute $\Phi_\ell$ over $\Z$ (via the CRT), or modulo an arbitrary integer $M$ (via the explicit CRT).
Our requirements for the order $\O$ and the primes $p$ are summarized in the definition below.

\begin{definition}\label{def:suitable}
Let $\ell>2$ be prime, and let $c>1$ be an absolute constant independent of $\ell$.
An imaginary quadratic order $\O$ is said to be \emph{suitable} for $\ell$ if $\ell\ndiv [\O_K:\O]$ and $\ell+2\le h(\O) \le c\ell$. A prime $p$ is then said to be \emph{suitable} for $\ell$ and $\O$ if $p\equiv 1\bmod \ell$ and $4p=t^2-\ell^2v^2\disc(\O)$, for some $t,v\in\Z$ with $\ell\ndiv v$.
\end{definition}

The definition of suitability above is weaker than that used in \cite{BLS}, but this only impacts logarithmic factors in the running time that are hidden by our soft asymptotic notation.

\subsection{Selecting primes with the GRH}\label{sec:primes}
In order to apply the isogeny volcano method to compute $\Phi_\ell$ (or the polynomials $H_D(F; x)$ we wish to compute), we need a sufficiently large set~$S$ of suitable primes $p$.
We deem $S$ to be sufficiently large whenever $\sum_{p\in S} \log p \ge B + \log 2$, where~$B$ is an upper bound on the logarithmic height of the integer coefficients that we wish to compute with the CRT.\footnote{When the coefficients are rational numbers that are not integers, we first clear denominators.}
For $\Phi_\ell(X,Y)=\sum_{i,j} a_{ij}X^iY^j$, we may bound $\height(\Phi_\ell):=\log\max_{i,j}|a_{ij}|$ using
\begin{equation}\label{eq:phiheight}
\height(\Phi_\ell) \le 6\ell\log \ell + 18\ell,
\end{equation}
as proved in \cite{BS}.

Heuristically (and in practice), it is easy to construct the set $S$.
Given an order~$\O$ of discriminant~$D$ suitable for $\ell$, we fix $v=2$ if $D\equiv 1\bmod 8$ and $v=1$ otherwise, and for increasing $t\equiv 2\bmod \ell$ of correct parity we test whether $p=(t^2-v^2\ell^2D)/4$ is prime.
We add each such prime to $S$, and stop when $S$ is sufficiently large.

Unfortunately, we cannot prove that this method will find \emph{any} primes, even under the GRH.
Instead, we use Algorithm~6.2 in~\cite{BLS}, which picks an upper bound~$x$ and generates random integers $t$ and $v$ in suitable intervals to obtain candidate primes $p=(t^2-v^2\ell^2D)/4\le x$ that are then tested for primality.
The algorithm periodically increases $x$, so its expected running time is $O(B^{1+\epsilon})$, even without the GRH.

Under the GRH, there are effective constants $c_1,c_2 > 0$ such that $x\ge c_1\ell^6\log^{4}{\ell}$ guarantees at least $c_2\ell^3\log^{3}{\ell}$ suitable primes less than $x$, by \cite[Thm.\ 4.4]{BLS}.
Asymptotically, this is far more than the $O(\ell)$ primes we need to compute $\Phi_\ell$.
We note that $S$ contains $O(B/\log B)$ primes (unconditionally), and under the GRH we have $\log p = O(\log B+\log\ell)$ for all $p\in S$.

\subsection{Modular singular moduli}

The results from the previous section can be cast in terms of the CM values of the $j$-function.
Indeed, we have the following classical theorem (for example, see \cite{Borel, Cox}) which summarizes
some of the most important properties of singular moduli for Klein's $j$-function.

\begin{theorem}\label{CMtheorem}
Suppose that $Q=ax^2+bxy+cy^2$ is a primitive positive definite binary
quadratic form with discriminant $D=b^2-4ac<0$, and let $\alpha_Q\in \H$
be the point for which $Q(\alpha_Q,1)=0$. Then the following are true:
\begin{enumerate}
\item The singular modulus $j(\alpha_Q)$ is an algebraic integer whose minimal polynomial has degree equal to the class number $h(D)$.
\item The Galois orbit of $j(\alpha_Q)$ consists of the $j(z)$-singular moduli
associated to the $h(D)$ equivalence classes in $\calQ_D^{\prim}/\SL_2(\Z)$.
\item If $K=\Q(\sqrt{D})$, then the discriminant $D$ singular moduli are conjugate over $K$.
Moreover, $K(j(\alpha_Q))$ is the ring class field of the quadratic order of discriminant $D$; in the case that $D$ is a fundamental discriminant, $K(j(\alpha_Q))$ is the Hilbert class field of~$K$.
\end{enumerate}
\end{theorem}

Theorem~\ref{CMtheorem} and the properties of the weight 2 nonholomorphic
Eisenstein series $E_2^{*}(z)$ at CM points shall play a central role in the construction of the
algorithms described in the next section. To this end, we make use of the special
nonholomorphic function $\gamma(z)$ defined in~(\ref{gammaz}).

Masser nicely observed that the singular moduli
for $\gamma(z)$ can be computed using the singular moduli for $j(z)$ and certain expressions for modular polynomials.
Here we make this precise for discriminants $D<-4$ that are not of the form $D=-3d^2$.
\begin{remark} Masser explains how to handle discriminants $D=-3d^2$; see p.\ 118 of \cite{Masser}.
\end{remark}

To state his observation, we let $\O$ be the imaginary quadratic order of discriminant $D$, and let $\{Q_1,\ldots,Q_h\}$ be a set of representatives for $\calQ_D^\prim/\SL_2(\Z)\simeq\cl(\O)$, where $h=h(D)$ is the class number.
To simplify notation, we use $\Phi_D$ to denote the classical modular polynomial $\Phi_{|D|}(X,Y)$.
For any $Q=Q_i$ we may write $\Phi_D$ in the form
\begin{equation}\label{eq:phiexp}
\Phi_D(X,Y) = \sum_{0\le \mu,\nu\le n} \beta_{\mu,\nu}\bigl(X-j(\alpha_Q)\bigr)^\mu\bigl(Y-j(\alpha_Q)\bigr)^\nu,
\end{equation}
where $n=\psi(|D|)$ is determined by the Dedekind $\psi$-function
\begin{equation}\label{eq:psi}
\psi(m) := m\prod_{p|m}(1+p^{-1}),
\end{equation}
which satisfies $\psi(m)=O(m\log\log m)$; see \cite{SP}.
The coefficients $\beta_{\mu,\nu}=\beta_{\mu,\nu}(\alpha_Q)$ are algebraic integers that lie in the ring class field $K_\O$, and we have $\beta_{\mu,\nu}=\beta_{\nu,\mu}$ (by the symmetry of $\Phi_D$).
Masser  \cite[p.\ 118]{Masser} gives the following formula for $\gamma(\alpha_Q)$.

\begin{lemma}\label{MasserLemma} Assuming the notation and hypotheses above, we have
\begin{equation}\label{eq:gammabeta}
\gamma(\alpha_Q) = \frac{2\beta_{0,2}(\alpha_Q)-\beta_{1,1}(\alpha_Q)}{\beta_{0,1}(\alpha_Q)}.
\end{equation}
\end{lemma}

\medskip\noindent
{\it Two remarks.}

\noindent
1) Masser proves (see
 \cite[Lemma A2]{Masser}) that $\beta_{0,1}(\alpha_Q)$ is nonzero.

\smallskip
\noindent 2)
From \eqref{eq:phiexp}, one finds that
\begin{align}\label{eq:betas}\notag
\beta_{0,1}(\alpha_Q) &= [Y]\Phi_D\bigl(j(\alpha_Q),Y+j(\alpha_Q)\bigr),\\
\beta_{1,1}(\alpha_Q) &= [Y]\Phi_D'\bigl(j(\alpha_Q),Y+j(\alpha_Q)\bigr),\\\notag
\beta_{0,2}(\alpha_Q) &= [Y^2]\Phi_D\bigl(j(\alpha_Q),Y+j(\alpha_Q)\bigr),
\end{align}
where $\Phi_D'(X,Y)=\frac{\partial}{\partial X}\Phi_D(X,Y)$, and for any polynomial $f(Y)$, the notation $[Y^k]f(Y)$ indicates the coefficient of $Y^k$ in $f(Y)$.

\section{The Algorithms}\label{alg}
Here we apply and extend the results in \S\ref{nutsandbolts} to derive our algorithms.

\subsection{Algorithm 1}\label{gamma_case}
We now give an algorithm to compute the class polynomial $H_D(\gamma;x)$, where $\gamma(z)$ is the nonholomorphic modular function defined in \eqref{gammaz} and $D$ is an imaginary quadratic discriminant.  In order to simplify the exposition as above, we shall assume $D<-4$ and that $D$ is not of the form $D=-3d^2$; these \emph{special} discriminants are in principle no more difficult to handle than the general case, but the details are more involved; see \cite[p.\ 118]{Masser}.\footnote{We note that the discriminants $D=1-24n$ needed to compute $H_n^\parti(x)$ are not special.}

To make use of Lemma~\ref{MasserLemma}, we need to compute the
singular moduli $j(\alpha_Q)$. These shall be obtained as the roots of the Hilbert class polynomial $H_D(x)$.
Thus  if we know $\Phi_D$ and $H_D$, then we can apply Lemma~\ref{MasserLemma} to compute
\[
H_D(\gamma;x) = \prod_{Q\in\calQ_D^{\prim}/\SL_2(\Z)}\bigl(x-\gamma(\alpha_{Q})\bigr).
\]

Using algorithms for fast multipoint polynomial evaluation and fast integer arithmetic (see \cite{GG}, for example), this yields an algorithm that computes $H_D(\gamma;x)$ in $\softO(|D|^3)$ expected time using $\softO(|D|^3)$ space, under the GRH.  However, this approach is quite memory intensive and quickly becomes impractical, even for moderate values of $D$.
As an alternative, we give a CRT-based algorithm that uses $\softO(|D|^{7/2})$ expected time and $\softO(|D|^2)$ space, under the GRH.\footnote{As remarked in the introduction, we expect this running time can be improved, possibly to $\softO(|D|^{5/2})$, by obtaining tighter bounds on the coefficients of $H_D(\gamma;x)$.}

It is clear from equations \eqref{eq:phiexp} and \eqref{eq:gammabeta} that the coefficients of $H_D(\gamma;x)$ lie in $\Q$; indeed, the coefficients of $\Phi_D$ are integers, as are the elementary symmetric functions of the $j(\alpha_Q)$, which are the coefficients of the Hilbert class polynomial $H_D(x)$. If we let
\begin{equation}
\delta:=\prod_{Q\in\calQ_D^{\prim}/\SL_2(\Z)}\beta_{0,1}(\alpha_Q),
\end{equation}
then $\delta\in\Z$ is divisible by the denominator of every coefficient of $H_D(\gamma;x)$ and $\delta H_D(\gamma;x)\in\Z[x]$.
We now present the algorithm.
\medskip

\noindent
\textbf{Algorithm 1}\\
\textbf{Input:} An imaginary quadratic discriminant $D$ that is not special.\\
\textbf{Output:} The polynomial $H_D(\gamma;x)\in\Q[x]$.\\
\vspace{-8pt}
\begin{enumerate}[1.]
\item Pick an order $\O$ suitable for $|D|$, and a set $S$ of primes suitable for $|D|$ and $\O$ (see Def.~\ref{def:suitablem}),
using the bound $B_\gamma(D)$ given in \eqref{def:Bgamma} below.
\item Compute the Hilbert class polynomial $H_D\in\Z[x]$ using \cite[Alg.\ 2]{Sutherland}.
\item For each prime $p\in S$:
\begin{enumerate}[a.]
\item Compute $\Phi=\Phi_D\bmod p$ using Algorithm 1.1 below.
\item Compute $\Phi'=\frac{\partial}{\partial X}\Phi_D(X,Y)\bmod p$.
\item Compute the roots $j_1,\ldots,j_h\in\Fp$ of $H_D\bmod p$.
\item Compute $\phi_k(Y)=\Phi(j_k,Y)$ and $\phi_k'(Y)=\Phi'(j_k,Y)$ for all $j_k$ using \cite[Alg.\ 10.7]{GG}.
\item For each $j_k$, compute $\beta_{0,1}, \beta_{1,1},$ and $\beta_{0,2}$ using $\phi_k$ and $\phi_k'$ via \eqref{eq:betas},\\
and then compute $\gamma_k=(2\beta_{0,2}-\beta_{1,1})/\beta_{0,1}$.
\item Compute $\delta=\prod_k \beta_{0,1}$ and $f(x)=\delta\prod_k(x-\gamma_k)$.
\item Save $f(x)\bmod p$ and $\delta\bmod p$.
\end{enumerate}
\item Use the CRT to recover $f(x)=\delta H_D(\gamma;x)\in\Z[x]$ and $\delta\in\Z$.
\item Output $H_D(\gamma;x) = \frac{1}{\delta}f(x)\in\Q[x]$.
\end{enumerate}
\medskip

Let $B_\Phi(D)$ denote an upper bound on $\height(\Phi_D)$; when $|D|$ is prime we may use the bound $B_\Phi(D)=6|D|+18|D|\log|D|$ from \eqref{eq:phiheight}, and otherwise we may derive such a bound by expressing $\Phi_D$ in terms of modular polynomials of prime level, as in \cite[Thm.~13.14]{Cox}.
We use
\begin{equation}
M(D):=\log\bigl(\exp(\pi\sqrt{|D|})+2114.567\bigr)
\end{equation}
to bound $\log|j(\alpha_Q)|$, for any $Q\in\calQ_D^\prim$; see \cite[p.~1094]{Enge}, for example.

We now define
\begin{equation}\label{def:Bgamma}
B_\gamma(D) := \bigl(h(D)+1\bigr)\Bigl(4\log\bigl(\psi(|D|)+1\bigr)+2\psi(|D|)M(D)+B_\Phi(D)+2\Bigr).
\end{equation}

\begin{lemma}\label{lemma:gammaheight}
Let $D$ be an imaginary quadratic discriminant that is not special.
Then we have the bounds $\delta \le B_\gamma(D)$ and $\height(\delta H_D(\gamma;x)) \le B_\gamma(D)$, and $\delta=\softO(|D|^{3/2})$.
\end{lemma}
\begin{proof}
For any $Q\in\calQ_D^\prim$, each coefficient $c$ of the univariate polynomial $\phi(Y)=\Phi_D(j(\alpha_Q),Y)$ is a polynomial of degree $\psi(|D|)$ in $j(\alpha_Q)$ with coefficients bounded by $\height(\Phi_D)\le B_\Phi(D)$.
It follows that
\begin{equation}\label{eq:gh1}
\height(\phi)\le \psi(|D|)M(D)+B_\Phi(D)+\log(\psi(|D|)+1).
\end{equation}
From \eqref{eq:betas}, we know that $\beta_{0,1}(\alpha_Q)$ is the linear coefficient of $\phi(Y+j(\alpha_Q))$, where $\phi(Y)$ has degree $\psi(|D|)$, and this implies
\begin{equation}\label{eq:gh2}
\log|\beta_{0,1}(\alpha_Q)|\le 2\log(\psi(|D|+1)+\height(\phi)+\psi(|D|)M
\end{equation}
for all $Q\in\calQ_D^\prim$.
Substituting \eqref{eq:gh1} into \ref{eq:gh2} and applying the bound to each of the $h(D)$ factors in the product $\delta=\prod_Q\beta_{0,1}(\alpha_Q)$ yields
\begin{equation}\label{eq:gh3}
\log|\delta|\le h(D)\bigl(3\log(\psi(|D|)+1)+2\psi(|D|)M(D)+B_\Phi(D)\bigr) \le B_\gamma(D).
\end{equation}

A calculation completely analogous to that used in \eqref{eq:gh2} yields
\begin{equation}\label{eq:gh4}
\log|2\beta_{0,2}(\alpha_Q)+\beta_{1,1}(\alpha_Q)|\le 3\log(\psi(|D|+1)+\height(\phi)+\psi(|D|)M + 2,
\end{equation}
for all $Q\in\calQ_D^\prim$.
The absolute values of the numerators of the coefficients of $H_D(\gamma;x)$, which has degree $h(D)$, have logarithms that exceed the bound in \eqref{eq:gh4} by at most $\log 2^{h(D)}$.  Combining this with \eqref{eq:gh3} yields
\begin{equation}
\height(\delta H_D(\gamma;x)) \le \log|\delta| + 3\log(\psi(|D|+1)+\height(\phi)+\psi(|D|)M + +h(D)\log 2 + 2,
\end{equation}
and it is then easy to check that plugging \eqref{eq:gh1} and \eqref{eq:gh3} into the RHS yields an expression that is bounded by $B_\gamma(D)$ as defined in \eqref{def:Bgamma}.  The asymptotic bound $B_\gamma(D)=\softO(|D|^{3/2})$ follows immediately from the bounds $h(D)=\softO(|D|^{1/2})$, $\psi(|D|)=\softO(|D|)$, $M(D)=\softO(|D|^{1/2})$ and $B_\Phi(D)=\softO(|D|)$.
\end{proof}

For an odd prime $\ell$, given a suitable order $\O$ and a suitable prime $p$,
the isogeny volcano algorithm of \cite[Alg.\ 2.1]{BLS} computes $\Phi_\ell\bmod p$ in $\softO(\ell^2)$ time, provided that $\log p = O(\log \ell)$.  Here we extend this result to any integer $m > 1$.
We first note that
\begin{align*}
\Phi_2(X,Y)= &X^3+Y^3-X^2Y^2+1488(X^2Y+XY^2)-162000(X^2+Y^2)\\
&+40773375XY+8748000000(X+Y)-157464000000000,
\end{align*}
and extend Definition \ref{def:suitable} to composite integers $m$.

\begin{definition}\label{def:suitablem}
Let $m>1$ be an integer, let $\ell$ be the largest prime divisor of $m$.
An imaginary quadratic order $\O$ is said to be \emph{suitable} for $m$ if $\psi(m)+1\le h(\O)\le 3\psi(m)$ and $\ell\ndiv [\O_K:\O]$.
A prime $p$ is \emph{suitable} for $m$ and~$\O$ if $p\equiv 1\bmod \ell$ and $4p=t^2- \ell^2v^2\disc(\O)$ for some $t,v\in\Z$ with $\ell\ndiv v$.
\end{definition}
\medskip

\noindent
\textbf{Algorithm 1.1}\\
\textbf{Input:} An integer $m>1$, an order $\O$ suitable for $m$, and a prime $p$ suitable for $m$ and $\O$.\\
\textbf{Output:} The modular polynomial $\Phi_m\bmod p$.\\
\vspace{-8pt}
\begin{enumerate}[1.]
\item If $m=2$ then output $\Phi_2\bmod p$ and terminate.
\item If $m$ is an odd prime then compute $\Phi_m\bmod p$ via \cite[Alg.~2.1]{BLS} and terminate.
\item Compute $\Phi_\ell$ for each prime $\ell\le \sqrt{m}$ dividing $m$.
\item Compute the Hilbert class polynomial $H_D$, where $D=\disc(\O)$, via \cite[Alg~2]{Sutherland}.
\item Compute the roots $j_1,\ldots,j_h\in\Fp$ of $H_D \bmod p$ and let $S=\{j_1,\ldots,j_n\}$, where $n=\psi(m)+1$.
\item If $m$ has a prime divisor $\ell_0>\sqrt{m}$ then compute $n$ sets of $j$-invariants $S_1^0,\ldots,S_n^0$,
where $S_i^0=\{\jt\in\Fp:\Phi_{\ell_0}(j_i,\jt)=0\}$, using the isogeny volcano method.\\
Otherwise, let $\ell_0=1$, and let $S_i^0=\{j_i\}$ for $1\le i\le n$.
\item Let $\ell_1\le \ell_2\le \cdots\le\ell_r$ be the primes whose product is $m/\ell_0$.\\
For $1\le d\le r$ do the following:
\begin{enumerate}[a.]
\item If $\ell_d\ne\ell_{d-1}$ then compute $S_i^d=\{\jt\in\Fp:\Phi_{\ell_i}(j_k,\jt) \text{ for some $j_k\in S_i^{d-1}$}\}$ for $1\le i\le n$.
\item If $\ell_d=\ell_{d-1}$ then compute $S_i^d=\{\jt\in\Fp:\Phi_{\ell_i}(j_k,\jt) \text{ for some $j_k\in S_i^{d-1}\backslash S_i^{d-2}$}\}$ for $1\le i\le n$.
\end{enumerate}
\item For $1\le i\le n$ compute $\phi_i(X)=\Phi_m(X,j_i)=\sum_{k=0}^{\psi(m)} a_{ik}X^k$ as the product $\prod_{\jt\in S_i^r}(X-\jt)$.
\item For $0\le k\le n$ interpolate the polynomial $f_k(X)$ of degree less than $n$ for which $f_k(j_i)=a_{ik}$.
\item Output $\Phi_m(X,Y) = \sum_{k=0}^{\psi(m)} f_kY^k\bmod p$.
\end{enumerate}
\medskip

\noindent
Note that step 6 does not use $\Phi_{\ell_0}$ to compute $S_i^0$, it uses the isogeny volcano method detailed in \cite[\S 6]{BLS},
whereas step 7 uses the (smaller) polynomials $\Phi_\ell$ computed in step 3.  When computing $\Phi_m\bmod p$ for many primes $p$ (as in Algorithm~1), the polynomials $H_D\in\Z[x]$ and $\Phi_\ell\in\Z[X,Y]$ computed in steps 3 and 4 may be computed just once and reused, since they do not depend on~$p$.

\begin{lemma}\label{lemma:alg11}
Algorithm~1.1 correctly computes $\Phi_m(X,Y)\bmod p$.
Under the GRH, its expected running time is $\softO(m^2)$, provided that $\log p = O(\log m)$.
\end{lemma}
\begin{proof}
For each of the $j$-invariants $j_i\in S$, the set $S_i^d$ contains all the $j$-invariants $\jt$ for which $\Phi_{m_d}(j_i,\jt)=0$, where $m_d = \prod_{i=0}^d \ell_d$.  This follows from the defining property of $\Phi_m(X,Y)$ (it parameterizes cyclic $m$-isogenies) and the fact that every (separable) cyclic isogeny can be expressed as a product of cyclic isogenies of prime degree (note that $p$ is distinct from all the $\ell_i$, since $p\equiv 1\bmod \ell_i$).  Thus we have $\phi_i(X)=\Phi_m(X,j_i)$ in step 8, and the $n=\psi(m)+1$ distinct values of $j_i\in S$ are sufficient to uniquely determine the coefficients of $\Phi_m\bmod p$ in step 9.

We now bound the complexity of each step, assuming the GRH and that $\log p = O(\log m)$.  Step 1 takes $\softO(1)$ time and step 2 takes $\softO(m^2)$ time, by \cite[Thm.\ 6.5]{BLS}.   Computing $O(\log m)$ modular polynomials $\Phi_\ell$ with $\ell\le \sqrt{m}$ in step 3 takes $\softO(m^{3/2})$ expected time, by \cite[Thm.\ 1]{BLS}.
We have $h=h(D)=O(m)$, since $\O$ is suitable for $m$, and since $h(D)=\softO(|D|^{1/2})$, we have $|D|=\softO(m^2)$ and this bounds the cost of step 3, by \cite[Thm.\ 1]{Sutherland}.
Using standard probabilistic algorithms for root-finding, step 5 takes $\softO(h)$ expected time, which is $\softO(m)$.

The cost of step 6 is bounded by $\softO(h\ell_0 + \ell_0^2)$, which is $\softO(m^2)$; this follows from the proof of \cite[Thm.\ 6.5]{BLS}.
Step 7 consists of root-finding operations in $\Fp$ whose expected complexity is softly-linear in the number of roots, ignoring factors of $\log p$ (every polynomial under consideration splits completely $\Fp[x]$ by virtue of the suitability of $\O$ and $p$).
The total number of roots computed in step 7 is $O(\psi(m)^2)$, hence the total cost is $\softO(m^2)$ expected time.

Using standard algorithms for fast arithmetic and polynomial interpolation \cite{GG}, the cost of steps 8 and 9 are both bounded by $\softO(\psi(m)^2)$, which is $\softO(m^2)$.  Thus every step has an expected running time bounded by $\softO(m^2)$.
\end{proof}

\begin{corollary}
Under the GRH, for any integer $m>1$ the modular polynomial $\Phi_m$ can be computed in $\softO(m^3)$ expected time.
\end{corollary}
\begin{proof}
An explicit $\softO(m)$ bound on the height of $\Phi_m$ can be derived from \cite[Prop.~13.14]{Cox} using the height bounds for $\Phi_\ell$ for primes $\ell|m$ given in \eqref{eq:phiheight}.
An order $\O$ suitable for $m$ can be obtained from the family of suitable orders given in \cite[Ex.~4.3]{BLS} and a sufficiently large set $S$ of primes $p$ suitable for $m$ and $\O$ can be selected using \cite[Alg.~6.2]{BLS}.
Under the GRH, these primes satisfy $\log p=O(\log m)$ and the corollary then follows from Lemma~\ref{lemma:alg11} and standard bounds on the time for fast Chinese remaindering \cite[\S10.3]{GG}.
\end{proof}

\begin{remark}
An algorithm to compute $\Phi_m$ using floating point approximations appears in \cite{Enge2} with a running time that is also $\tilde{O}(m^3)$, but the correctness of this algorithm and the bound on its running time both depend on a heuristic assumption regarding the precision needed to avoid rounding errors.  We note that Algorithm 1.2 is faster in practice, its output is provably correct, and the bound on its expected running time depends only on the GRH.
\end{remark}

\begin{corollary}
Let $D$ be an imaginary quadratic discriminant that is not special, let $\O$ be an order suitable for $|D|$ with $h(D)=O(|D|)$, and let $p$ be a prime suitable for $|D|$ and $\O$.
Under the GRH, the polynomial $H_D(\gamma;x)\bmod p$ can be computed in $\softO(|D|^2)$ expected time.
\end{corollary}
\begin{proof}
Apply Lemma~\ref{lemma:alg11} to step 3 of Algorithm~1.
\end{proof}

We now prove Theorem~\ref{gamma_thm} given in the introduction, which we restate here.

\begin{gamma_thm}
Algorithm~1 computes $H_D(\gamma;x)$.
Under the GRH, its expected running time is $\softO(|D|^{7/2})$ and it uses $\softO(|D|^2)$ space.
\end{gamma_thm}
\begin{proof}
The correctness of Algorithm~1 follows from the discussion preceding the algorithm, which shows how to compute $H_D(\gamma;x)$ in terms of $\Phi_D$ and the Hilbert class polynomial $H_D$, the correctness of the algorithm used to compute $H_D$ \cite[Thm.\ 1]{Sutherland}, the correctness of Algorithm~1.1 used to compute $\Phi_D\bmod p$ (Lemma~\ref{lemma:alg11}), and the validity of the bound $B_\gamma(D)$ on the logarithmic height of $H_D(\gamma;x)$ (Lemma~\ref{lemma:gammaheight}).

For the time and space bounds, we now assume the GRH.  We first note that, as explained in \S\ref{sec:primes}, we can select the set of primes $S$ in step 1 in $O(B_\gamma(D)^{1+\epsilon})$ time, which is $O(|D|^{3/2+\epsilon})$ for any $\epsilon>0$, since $B_\gamma(D)$ is $\softO(|D|^{3/2})$, by Lemma~\ref{lemma:gammaheight}, and the primes $p\in S$ all satisfy $\log p=O(\log |D|)$.
The time to compute the Hilbert class polynomial $H_D$ in step 2 is $\softO(|D|)$, by \cite[Thm.\ 1]{Sutherland}, and its size is $\softO(|D|)$.
The set $S$ has cardinality $O(B_\gamma(D))$, which is $\softO(|D|^{3/2})$, and each iteration of step 3 takes $\softO(|D|^2)$ expected time: this follows from Lemma~\ref{lemma:alg11}, the time for fast multipoint polynomial evaluation \cite[Cor.\ 10.8]{GG}, and standard bounds on the complexity of fast arithmetic in $\Fp[x]$.  Thus the total expected time for step 3 is $\softO(|D|^{7/2})$.

The space used in each iteration of step 3 must be bounded by $\softO(|D|^2)$, since this bounds the time, and the total size of the values $f(x)\bmod p$ and $\delta\bmod p$ saved is bounded by $O(h(D)B_\gamma(D))$, which is $\softO(|D|^2)$.
Finally, with fast Chinese remaindering \cite[Alg.\ 10.22]{GG}, the cost of step 4 is softly-linear in the total size of the coefficients of $\delta H_D(\gamma;x)$ and $\delta$, which is $\softO(|D|^2)$.
\end{proof}

\subsection{Algorithm~2}\label{good_case}

We now give an algorithm to compute the class polynomial $H_D(F;x)$ for a good modular function $F(z)=\sum A_n(z)\gamma(z)^n$, as defined in the introduction.
We assume that each $A_n(z)$ is written in the form $A_n(z)=r_n(j(z))$, where $r_n\in\Z(x)$.
\medskip

\noindent
\textbf{Algorithm~2}\\
\textbf{Input:} An imaginary quadratic discriminant $D$ that is not special.\\
\textbf{Output:} The polynomial $H_D(F;x)\in\Q[x]$.\\
\vspace{-8pt}
\begin{enumerate}[1.]
\item Pick an order $\O$ lying in the order of discriminant $D$ that is also suitable for  $|D|$, and a set $S$ of primes suitable for $|D|$ and $\O$ (see Def.~\ref{def:suitablem}) such that no prime in $S$ divides the denominator of any of the $r_n(x)$,
using the height bound $B_F(D)$ (discussed below).
\item Compute the Hilbert class polynomial $H_D\in\Z[x]$ using \cite[Alg.\ 2]{Sutherland}.
\item For each prime $p\in S$:
\begin{enumerate}[a.]
\item Compute $\Phi_D\bmod p$ using Algorithm 1.1.
\item Compute the roots $j_1,\ldots,j_h\in\Fp$ of $H_D\bmod p$.
\item For each $j_k$ do the following:
\begin{enumerate}[i.]
\item Compute $\gamma_k=(2\beta_{0,2}-\beta_{1,1})/\beta_{0,1}\bmod p$ as in Algorithm 1.
\item Compute $F_k = \sum r_n(j_k)\gamma_k^n\bmod p$.
\end{enumerate}
\item Compute $f(x)=c_1|D|^{c_2h}\prod_k(x-F_k)\bmod p$.
\item Save $f(x)\bmod p$.
\end{enumerate}
\item Use the CRT to recover $f(x)=c_1|D|^{c_2h(D)} H_D(F;x)\in\Z[x]$.
\item Output $H_D(F;x) = \frac{1}{c_1}|D|^{-c_2h(D)}f(x)\in\Q[x]$.
\end{enumerate}
\medskip

The bound $B_F(D)$ used in step 1 is an upper bound on $\height\left(c_1|D|^{c_2h(D)}H_D(F;x)\right)$, which the next result shows
is  $\softO(|D|^{1/2})$.
Explicit computation of $B_F(D)$ depends on the particular functions $A_n(z)$;  bounds on the heights of the class polynomials $H_D(A_n;x)$ can be readily derived from the functions $r_n(x)$ and known bounds on the height of the Hilbert class polynomial $H_D$; see Lemma~8 in \cite{Sutherland}, for example.
From these, one can derive an explicit bound $B_F(D)$ on the height of $H_D(F;x)$; see Lemma~\ref{lemma:partht} in the next section for an example.
In general, the following lemma gives us an asymptotic bound for $B_F(D)$ that suffices to bound the complexity of Algorithm~2.

\begin{lemma}\label{lemma:alg2ht}
For all non-special imaginary quadratic discriminants $D$ we have
\[
\height\left(c_1|D|^{c_2h(D)}H_D(F;x)\right) = \softO(|D|^{1/2}).
\]
\end{lemma}
\begin{proof} The proof follows as in the proof of Lemma~8 of \cite{Sutherland}. One only needs to take care of the
dependence on the summand $\frac{3}{\pi \im(z)}$ in the definition of $E_2^*(z)$ which in turn appears in the definition
of $\gamma(z)$. We leave these details to the reader.
\end{proof}

We now prove Theorem~\ref{general_thm} given in the introduction, which we restate here.

\begin{general_thm} For discriminants $D<-4$ not of the form
$D=-3d^2$, Algorithm~2 computes $H_D(F;x)$ for good modular functions~$F(z)$.  Under the GRH, its expected running time is $\softO(|D|^{\nicefrac{5}{2}})$ and it uses $\softO(|D|^2)$ space.
\end{general_thm}
\begin{proof}
The correctness of Algorithm~2 is clear.  We now bound its complexity, under the GRH.
By Lemma~\ref{lemma:alg2ht}, the set $S$ contains $\softO(|D|^{1/2})$ primes, and as described in \S\ref{sec:primes}, we can construct~$S$ in $\softO(|D|^{1/2})$ expected time.  The expected time to compute $H_D(X)$ in step 2 is $\softO(|D|)$, by \cite[Thm.\ 1]{Sutherland}.
Each of the primes $p\in S$ satisfies $\log p =O(\log|D|)$, and therefore the expected time for step 3a is $\softO(|D|^2)$, by Lemma~\ref{lemma:alg11}.
This dominates the cost of steps 3b and 3c, and the total expected time for step 3 is thus $\softO(|D|^{5/2})$, which dominates the expected running time of the entire algorithm.
The space complexity of step 3 is dominated by the size of $\Phi_D\bmod p$, which is $\softO(|D|^2)$.
\end{proof}

\subsection{Algorithm 3}\label{P_case}

We now give an algorithm to compute the partition polynomial
\[
H_{n}^{\parti}(x):=\prod_{Q\in \calQ_{6,1-24n,1}}(x-P(\alpha_Q))
\]
defined in \eqref{partitionpolynomial}.
We do this by expressing $H_{n}^{\parti}(x)$ as a product of class polynomials
\begin{equation}
\label{classpolysp0}
H_{D,\beta}(P;x):=\prod_{Q\in \calQ_{6,D,\beta}^{\prim}/\Gamma_0(6)}\bigl(x-P(\alpha_Q)\bigr).
\end{equation}

\begin{lemma}\label{lemma:twisted}
For a positive integer $n$, let $D=1-24n$.  Then
\[
H_{n}^{\parti}(x) = \prod_{\substack{u>0\\u^2|D}} \varepsilon(u)^{h(D/u^2)}H_{D/u^2,1}(P;\varepsilon(u)x),
\]
where the class polynomials on the right and side are defined by \eqref{classpolysp0}, and  $\varepsilon(u)=1$ if $u\equiv \pm 1\bmod 12$ and $\varepsilon(u)=-1$ otherwise.
\end{lemma}

\begin{proof}
Using \cite{GKZ}, p.~505 equation (1), we obtain
\begin{equation}
\label{classpolysp}
H_{n}^{\parti}(x)=\prod_{\substack{u>0\\u^2\mid 1-24n}} H_{(1-24n)/u^2,\beta_u}(P;x),
\end{equation}
where $\beta_u\in \Z/12\Z$ denotes the unique residue such that $\beta_u\cdot u\equiv 1 \pmod{12}$, equivalently, $\beta_u\equiv u\pmod{12}$.

Let $\Delta<0$ be any discriminant such that  $\Delta \equiv 1\pmod{24}$.
Since $P$ is invariant under the Atkin-Lehner involution $W_6$, we have
\[
H_{\Delta,-1}(P;x)=H_{\Delta,1}(P;x).
\]
Since $P$ is taken to its negative under the Atkin-Lehner involution $W_3$, we have
\[
H_{\Delta,\pm 5}(P;x)=H_{\Delta,1}(-P;x)=(-1)^{h(\Delta)}H_{\Delta,1}(P;-x).
\]
Putting this into \eqref{classpolysp}, we obtain the assertion.
\end{proof}

For the remainder of this section (and also in Section~4), we shall abuse notation and drop the dependence on $\beta$ for modular functions on
$\Gamma_0(6)$. In every case we will have $\beta=1$. For example, we let
$H_{D}(P;x):=H_{D,1}(P;x)$ for convenience.
We cannot directly apply Algorithm~2 to compute $H_{D}(P;x)$ because the function $P(z)$ does not satisfy all of our requirements for a good modular function; some minor changes are required, as we now explain.

As shown by Larson and Rolen, the function $P(z)$ may be decomposed as
\begin{equation}\label{eq:PAB}
P(z) = A(z)+B(z)\gamma(z),
\end{equation}
where $\widehat{A}(z)=A(z)j(z)(j(z)-1728)$ and $B(z)$ are weakly holomorphic modular functions for $\Gamma_0(6)$; see Lemma 2.2 in \cite{LarsonRolen}.
The expression for $P(z)$ in \eqref{eq:PAB} does not satisfy our definition of a good modular function $F(z)$ because $A(z)$ and $B(z)$ are not rational functions of $j(z)$.  However our two key requirements are satisfied:
for discriminants $D$ of the form $1-24n$, the values of $A(z)$ and $B(z)$ at CM points lie in the ring class field $K_\O$, and the polynomial $|D|H_D(P;x)$ has integer coefficients (so $c_1=c_2=1$).

For each of the functions $g=\widehat{A},B$, Larson and Rolen compute explicit polynomials
\[
\Psi_g(X;z) :=\prod_{\alpha} (X-g(\alpha z)),
\]
where the product ranges over right coset representatives $\alpha$ for $\Gamma_0(6)$ in $\SL_2(\Z)$.
The polynomials $\Psi_g$ may be expressed as polynomials in $X$ whose coefficients are integer polynomials in $j(z)$, and we regard them as elements of $\Z[X,J]$; see Appendix A of \cite{LarsonRolen} for the exact values of $\Psi_{\widehat{A}}$ and $\Psi_B$.  
Here each occurrence of $j(z)$ is replaced by the indeterminate $J$.
While $\widehat{A}(z)$ and $B(z)$ are not rational functions of $j(z)$, we note that the curves defined by $\Psi_{\widehat{A}}(X,J)$ and $\Psi_B(X,J)$ both have genus 0 (and thus admit a rational parametrization, although we shall not make explicit use of this fact).

It follows that, at least for discriminants prime to 6, the CM values of $\widehat{A}(z)$ and $B(z)$ are class invariants, and we can compute the class polynomials $H_D(\widehat{A};x)$ and $H_D(B;x)$ using standard algorithms such as those found in \cite{Broker,Enge,EngeSutherland}.
Under the GRH, we can compute these class polynomials in $\softO(|D|)$ expected time, which is negligible compared to the $\softO(|D|^{5/2})$ expected running time of Algorithm~2.

For $g=\widehat{A},B$,  we can use $H_D(g;x)$ to uniquely determine a root $g_k$ of $\Psi_g(x,j_k)$ corresponding to a singular modulus $j_k$ by computing the unique root of the linear polynomial
\[
f_k(g;x) := \gcd\bigl(\Psi_g(x,j_k),H_D(g;x)\bigr).
\]
This is useful because we would otherwise have 6 possible values of $g_k$ to choose from; in both cases $\Psi_g(x,j_k)$ has degree 12, and 6 of its roots lie in the ring class field.  In the context of Algorithm~2, we can use the values $g_k$ to replace the quantities $r_n(j_k)$ in step 3.c.ii that require $A_n(z)$ to be a rational function of $j(z)$.  For this purpose we let $\hat a_k$ and $b_k$ denote the unique roots of the polynomials $f_k(\widehat{A};x)$ and $f_k(B;x)$, respectively, and let $a_k$ denote $\hat a_k/(j_k(j_k-1728))$.

There is one other issue to consider.
The coefficients of the class polynomials $H_D(\widehat{A};x)$ and $H_D(B;x)$ are not rational integers; they are algebraic integers in the quadratic field $K=\Q(\sqrt{D})$.
This presents a potential difficulty for the CRT approach;
while we always work modulo primes $p$ for which $D$ is a quadratic residue, we must make an arbitrary choice for the square root of $D$ modulo $p$, and there is no clear way to make these choices consistently across many primes~$p$.
The following lemma implies that it does not matter which choice we make, we will get the same answer in either case.

\begin{lemma}\label{lemma:involution}
Let $P(z)=\widehat{A}(z)/(j(z)(j(z)-1728) + B(z)\gamma(z)$ be as above.
Then for any discriminant $D=1-24n$, we have
\[
\{P(\alpha_Q):\alpha_Q\in\calQ_{6,D,\beta}^{\prim}/\Gamma_0(6)\} =
\{ \overline{P(\alpha_Q)}:\alpha_Q\in\calQ_{6,D,\beta}^{\prim}/\Gamma_0(6)\}.
\]
\end{lemma}
\begin{proof}
For any modular function $f$ as $P$
(e.g. any weak Maass form) with real coefficients, we have
\[
\{(f\mid W_6)(\alpha_Q):\alpha_Q\in\calQ_{N,D,\beta}^{\prim}/\Gamma_0(6)\} =
\{ \overline{f(\alpha_Q)}:\alpha_Q\in\calQ_{N,D,\beta}^{\prim}/\Gamma_0(6)\}.
\]
Using the invariance of $P$ under $W_6$, we get
\[
\{P(\alpha_Q):\alpha_Q\in\calQ_{6,D,\beta}^{\prim}/\Gamma_0(6)\} =
\{ \overline{P(\alpha_Q)}:\alpha_Q\in\calQ_{6,D,\beta}^{\prim}/\Gamma_0(6)\}.
\]
\end{proof}

We now give the algorithm to compute the partition polynomial $H_{n}^{\parti}(x)$.
\medskip

\noindent
\textbf{Algorithm 3}\\
\textbf{Input:} A positive integer $n$.\\
\textbf{Output:} The partition polynomial $H_{n}^{\parti}(x)\in\Q[x]$.\\
\vspace{-8pt}
\begin{enumerate}[1.]
\item Let $1-24n=v^2D_0$, where $D_0$ is a fundamental discriminant.
\item For each divisor $u$ of $v$, let $D=u^2D_0$ and compute $H_{D}(P;x)$ as follows:
\begin{enumerate}[a.]
\item Compute the class polynomials $H_{D}(\widehat A;x)$ and $H_{D}(B;x)$.
\item Using the height bound $B_P(D)$ defined below, compute the polynomial $H_{D}(P;x)$ using a modified version of Algorithm~2 in which $r_0(j_k)$ is replaced by $a_k=\widehat a_k/(j_k(j_k-1728))$ and $r_1(j_k)$ is replaced by $b_k$, where $a_k$ and $b_k$ are as defined above.
\end{enumerate}
\item Compute $H_{n}^{\parti}(x)\in\Q[x]$ via Lemma~\ref{lemma:twisted}.
\end{enumerate}
\medskip

The height bound $B_P(D)$ is defined as
\begin{equation}\label{bounds}
  B_P(D):=c_1 B_j(D)+h(D)\log |D|,
\end{equation}
where $B_j(D)$ is an explicit bound on the height of the Hilbert class polynomial derived as in Lemma~8 of \cite{Sutherland}.
Here $c_1$ denotes an effectively computable positive constant.

\begin{remark}
We have not tried to obtain the optimal constant for $c_1$.
However, it is reasonable to suspect that we can take $c_1:=7/3$, which we note is equal to $\deg_J(\Psi_{\widehat A})/\deg_X(\Psi_{\widehat A})=28/12$ (which dominates $\deg_J(\Psi_B)/\deg_X(\Psi_B)=18/12$).
\end{remark}

\begin{lemma}\label{lemma:partht}
For all discriminants $D\equiv 1\bmod 24$ we have $\height\bigl(H_D(P;x)\bigr) \le B_P(D)$.
\end{lemma}
\begin{proof}
The proof is analogous to
 the proof of Lemma \ref{lemma:alg2ht}, which in turn follows as in the proof of Lemma~8 of \cite{Sutherland}.
Decorating that  proof with the asymptotic properties of the Fourier expansion of the function $P(z)$,
 which is the image of a simple
weight $-2$ weakly holomorphic modular form under the differential operator
$\partial_{-2}:=\frac{1}{2\pi i}\cdot \frac{\partial}{\partial z}+\frac{1}{2\pi \im(z)}$,
gives the desired result.
\end{proof}

We now prove Theorem~\ref{partition_thm} given in the introduction, which we restate here.

\begin{partition_thm} For all positive integers $n$, Algorithm~3 computes $H_{n}^{\parti}(x)$.
Under the GRH, its expected running time is $\softO(n^{\nicefrac{5}{2}})$ and it uses $\softO(n^2)$ space.
\end{partition_thm}
\begin{proof}
The correctness of Algorithm~3 follows from Lemmas~\ref{lemma:twisted} and~\ref{lemma:involution}, and the correctness of Algorithm~2.
For the complexity bound, we first note that the degree of $H_{n}^{\parti}(x)$ is given by the Hurwitz-Kronecker class number $H(1-24n)=\sum_D h(D)$, where $D$ varies over negative discriminants dividing $1-24n$.
It is known that $H(D)=O\bigl(h(D)(\log\log|D|)^2\bigr)$; see, e.g., \cite[Lemma~9]{Sutherland}.
The complexity bounds then follow from the height bound in Lemma~\ref{lemma:partht} and the complexity bounds in Theorem~\ref{general_thm}.
\end{proof}

To simplify the practical implementation of Algorithm~3, we make the following remark:
it is not actually necessary to compute the class polynomials $H_{D}(\hat{A};x)$ and $H_{D}(B;x)$.
Instead, for each singular modulus $j_k$, one can simply compute all 36 possible combinations $s_i+t_ij_k$, where $s_1,\ldots,s_6$ are the roots of $\Psi_{\hat{A}}(x,j_k)$ that lie in $\Fp$ (where $p\equiv 11\bmod 12$ is a suitable prime) and $t_1,\ldots,t_6$ are the roots of $\Psi_B(x,j_k)$ that lie in $\Fp$.
For all but finitely many primes $p$, exactly 32 of these 36 values will be distinct, and there will be two pairs of repeated values, corresponding to $P_k=a_k+b_kj_k$ and $-P_k=-a_k-b_kj_k$.  We do not prove this claim here, but observe that Lemma~\ref{lemma:involution} guarantees that the value $P_k$ will be repeated, so if one in fact finds the situation modulo $p$ to be as claimed (exactly two pairs of repeated values that differ only in sign), then the end result will be provably correct.  For the handful of primes $p$ where the claim does not hold, one simply discards $p$ and selects another suitable prime in its place.

Using the observation above, one may compute the polynomial
\[
f(x)=|D|^{2h(D)}\prod_k(x^2-P_k^2) = (-1)^{h(D)}|D|^{2h(D)}H_{D}(P;x)H_{D}(P;-x) \bmod p
\] modulo a sufficient number of primes $p$ (using a suitably increased height bound), and then apply the CRT to obtain the integer polynomial $f(x)$ which may then be factored in $\Z[x]$ to yield the required polynomial $H_{D}(P;x)$.


\section{Numerical examples}\label{example}

As a first example, let us compute $H_1^{\parti}(x)$ using Algorithm~3, recapitulating the example given in the introduction of \cite{BruinierOno2011}.
We have $n=1$, and the discriminant $D=1-24n=-23$ is fundamental, so $u=1$.  We begin by computing the class polynomials
\footnotesize
\begin{align*}
H_{-23}(\widehat A; x) &= x^3 + (264101659831625\Delta - 76898070951625)/2\thinspace x^2\\
&\hspace{23.6pt}+ (4866595720359935196250\Delta + 1237728700002625503750)\thinspace x\\
&\hspace{23.6pt}+ (-14048754886813637262794029921875\Delta - 31056014444792221417574181765625)/2,\\
H_{-23}(B;x) &=x^3 + (-35487375\Delta - 35487375)\thinspace x^2 + (6837889760625\Delta - 75216787366875)/2\thinspace x\\
&\hspace{23.6pt} + (842331597312734375\Delta + 2863927430863296875)/2,
\end{align*}
\normalsize
where $\Delta$ denotes a square root of $-23$.
We then compute the {\it heuristic} height bound
\[
B_P(-23) = \nicefrac{7}{3}\ B_j(-23) + h(-23)\log23\approx 83.25
\]
for $H_{-23}(P;x)$, using \cite[Lemma~8]{Sutherland} to compute $B_j(-23)\approx 31.65$.

\begin{remark} We refer to this bound as a heuristic bound because we used $c_1=7/3$ as
in the discussion in the previous section.
Implementing the algorithm with this bound for $n\leq 750$ always gave the correct values for $p(n)$.
Moreover, in every case the polynomials $H_D(P;x)$ computed by the algorithm split into linear factors over the ring class field for the order for discriminant $D$.
\end{remark}

We now use Algorithm~2 to compute $H_{-23}(P;x)$, with $r_0(j_k)=\widehat a_k/(j_k(j_k-1728))$ and $r_1(j_k)=b_k$.
The order $\O$ of index 11 in the order of discriminant $-23$ is suitable for the integer 23, and the primes in the set
\[
S=\{1562207,2744591,4294607,6454031,7089107,10010291\}
\]
are all suitable for $\O$ and $23$,  with $\prod_{p\in S}\log p \approx 91.93 > B_P(-23) + \log 2$.
We then compute
\[
H_{-23}(j;x) =  x^3 + 3491750 x^2 - 5151296875 x + 12771880859375
\]
using \cite[Alg.\ 2]{Sutherland}.

Starting with the first prime $p=1562207$, we compute $\Phi_{23}\bmod p$ using Algorithm 1.1 (which in this case just calls \cite[Alg.\ 2.1]{BLS}, since $23$ is prime).
We then find the roots $j_k$ of $H_{-23}(x)\bmod p$, and for each $j_k$ we compute:
\begin{itemize}
\item $\gamma_k=(2\beta_{0,2}-\beta_{1,1})/\beta_{0,1}\bmod p$ (via \eqref{eq:betas}, using $\Phi_{23}\bmod p$ and $j_k$).
\item $\widehat a_k$ as the unique root of $f_k(\hat A; x)=\gcd\bigl(H_{-23}(\hat A;x),\Psi_{\widehat A}(x,j_k)\bigr)\bmod p$.
\item $b_k$ as the unique root of $f_k(B;x)=\gcd\bigl(H_{-23}(B;x),\Psi_B(x,j_k)\bigr)\bmod p$.
\item $P_k = a_k/(j_k(1728-j_k)) + b_k\gamma_k$.
\end{itemize}

For the prime $p=1562207$ the results of these computations are summarized below.

\begin{center}
\begin{tabular}{rrrr}
&\underline{$\quad k=1$}&\underline{$\quad k=2$}&\underline{$\quad k=3$}\\
$j_k$:&244476&467416&482979\\
$\gamma_k$:&1461486&587848&220836\\
$\widehat a_k$:&1201792&98544&239915\\
$b_k$:&1120135&560362&531933\\
$P_k$:&1352290&519913&1252234
\end{tabular}
\end{center}
\smallskip

Using the constants $c_1=c_2=1$ in our definition of a good modular function, we compute
\[
f(x) = 23^3\prod_k(x-P_k) = 12167x^3 + 1282366x^2 + 337961x + 1150855 \pmod{1562207}
\]
as the reduction of $|D|^{h(D)}H_D(P;x)$ modulo $p$.
Repeating this process for the remaining primes in $S$ yields the following polynomials $|D|^{h(D)}H_D(P;x)\bmod p$.
\smallskip

\begin{center}
\begin{tabular}{ll}\vspace{2pt}
$12167x^3 + 2464750x^2 + 1900168x + 391209$&$\pmod{2744591}$\\\vspace{2pt}
$12167x^3 + 4014766x^2 + 1900168x + 3491241$&$\pmod{4294607}$\\\vspace{2pt}
$12167x^3 + 6174190x^2 + 1900168x + 1356058$&$\pmod{6454031}$\\\vspace{2pt}
$12167x^3 + 6809266x^2 + 1900168x + 1991134$&$\pmod{7089107}$\\\vspace{2pt}
$12167x^3 + 9730450x^2 + 1900168x + 4912318$&$\pmod{1001029}$
\end{tabular}
\end{center}

Applying the Chinese remainder theorem, and using the fact that $\prod_{p\in S}p$ is more than twice the absolute value of the largest coefficient of $23^3H_{-23}(P;x)$, we obtain
\[
23^3 H_{-23}(P;x) = 12167x^3 - 279841x^2 + 1900168x - 5097973 \in \Z[x].
\]
We note that the coefficients of the above polynomial are all much smaller than the product
$$\prod_{p\in S} p = 5398465666938830659283417535896039,$$
indicating that our height bound $B_P(-23)$ is actually bigger than it needs to be.
Dividing by $23^3$ and applying Lemma~\ref{lemma:twisted} we obtain
\[
H_1^{\parti}(x) = H_{-23}(P;x) = x^3 - 23x^2 + \frac{3592}{23}x - 419,
\]
which completes the execution of Algorithm~3 for $n=1$.  As proven in \cite{BruinierOno2011}, if we divide the trace of $H_n^{\parti}(x)$ by $24n-1$, we obtain the $n$th partition number $p(n)$.  In this case we have $23/23=1=p(1)$.

We now consider the case $n=24$, which is  the least $n$ for which $1-24n$ is not a fundamental discriminant.
We have $1-24\cdot 24 =-575= -5^2\cdot 23$, and Lemma~\ref{lemma:twisted} then implies that
\begin{equation}\label{eq:H24}
H_{24}^{\parti}(x) = -H_{-23}(P;-x)H_{-575}(P;x).
\end{equation}
We have already computed $H_{-23}(P;x)$, and in the same way we may compute
\footnotesize
\begin{align*}
H_{-575}(P;x) &= x^{18} - 905648\thinspace x^{17} + 7864919720287/23\thinspace x^{16} - 62085428963462224\thinspace x^{15}\\
&\hspace{27.5pt}+ 2500819220800663290310031/529\ x^{14} - 145570369368132345878793951/23\thinspace x^{13}\\
&\hspace{27.5pt}\cdots\\
&\hspace{27.5pt} + 758005997309239141979280480729944052789478182183267952/3700897225\thinspace x\\
&\hspace{27.5pt} -274989755819545226019386671943056995003866543720439419/18504486125.
\end{align*}
\normalsize
From \eqref{eq:H24} we obtain
\footnotesize
\begin{align*}
H_{24}^{\parti}(x) &= x^{21} - 905625\thinspace x^{20} + 341932201569\thinspace x^{19}  - 62077564185180110\thinspace x^{18}\\
&\hspace{27.5pt} + 2500063855637055742916679/529\thinspace x^{17}- 143069773154897117981992275/23\thinspace x^{16}\\
&\hspace{27.5pt} -  248682508073724592034185083695904/60835\thinspace x^{15}\\
&\hspace{27.5pt} + 4721274513295479628753048946698042/2645\thinspace x^{14} \\
&\hspace{27.5pt} - 684240866701755248448205419660018178147/1399205\thinspace x^{13}\\
&\hspace{27.5pt} + 828297525091153912001188772487055395656/12167\thinspace x^{12}\\
&\hspace{27.5pt} - 32704304695374273471069347508729088366971453/6436343\thinspace x^{11}\\
&\hspace{27.5pt} + 290553028842402057481729080665422874771306601/1399205\thinspace x^{10}\\
&\hspace{27.5pt} - 15618334996574598433984982031615985504271825288372/3700897225\thinspace x^{9}\\
&\hspace{27.5pt} + 2971138261271289839650966142959376571788416952712/160908575\thinspace x^{8}\\
&\hspace{27.5pt} + 67822191247241980381807708488720865403444300542792174/85120636175\thinspace x^{7}\\
&\hspace{27.5pt} - 10287891953477631667871642653944942982233172929865507/740179445\thinspace x^{6}\\
&\hspace{27.5pt} + 120072172960067820695115892912976299403813878193923504758/1957774632025\thinspace x^{5}\\
&\hspace{27.5pt} + 9442155332145807613622010202526881668517330792046133529/17024127235\thinspace x^{4}\\
&\hspace{27.5pt} - 944566531689753532003676376487531915501990271825184156855477/225144082682875\thinspace x^{3}\\
&\hspace{27.5pt} - 512515146501467199140764542151150418963279118308213518346717/9788873160125\thinspace x^{2}\\
&\hspace{27.5pt} - 35536755777441881604409993038352893457607117583456947874072/425603180875\thinspace x\\
&\hspace{27.5pt} +115220707688389449702123015544140880906620081818864116561/18504486125.
\end{align*}
\normalsize

If we divide the trace $905625$ of $H_{24}^{\parti}(x)$  by $24\cdot 24-1=575$ we obtain the partition number $p(24)=1575$, as expected.  A  table of all the polynomials $H_{n}^{\parti}(x)$ for $n\le 750$ is available online at
\url{http://math.mit.edu/~drew}.

\end{document}